\documentclass[12pt,reqno]{amsart}

\usepackage[arrow,matrix,curve]{xy}

\usepackage[dvips]{graphicx} 

\usepackage{amssymb, latexsym, amsmath, amscd, array, hyperref
}

\usepackage{breakurl}   

\newtheorem{theorem}{Theorem}[section]

\newtheorem{corollary}[theorem]{Corollary}

\theoremstyle{definition}
\newtheorem{definition}[theorem]{Definition}

\numberwithin{equation}{section}

\newcommand\N {{\mathbb N}} 

\newcommand\R {{\mathbb R}}

\numberwithin{equation}{section}

\title[Infinitesimals via Cauchy sequences: Refining equivalence]
{Infinitesimals via Cauchy sequences: Refining the classical
equivalence}

\author{Emanuele Bottazzi} \address{E.~Bottazzi, Department of
Mathematics F. Casorati, University of Pavia, Via Ferrata, 5, 27100
Pavia, Italy} \email{emanuele.bottazzi@unipv.it,
emanuele.bottazzi.phd@gmail.com}

\author{Mikhail G. Katz} \address{M.~Katz, Department of Mathematics,
Bar Ilan University, Ramat Gan 5290002 Israel}
\email{katzmik@math.biu.ac.il}

\begin{document}

\thispagestyle{empty}


\begin{abstract}
A refinement of the classic equivalence relation among Cauchy
sequences yields a useful infinitesimal-enriched number system.  Such
an approach can be seen as formalizing Cauchy's sentiment that a null
sequence ``becomes'' an infinitesimal.  We signal a little-noticed
construction of a system with infinitesimals in a 1910 publication by
Giuseppe Peano, reversing his earlier endorsement of Cantor's
belittling of infinitesimals.
\end{abstract}


\maketitle

\today

\section
{Historical background}
\label{one}

Robinson developed his framework for analysis with infinitesimals in
his 1966 book~\cite{Ro66}.  There have been various attempts either
\begin{enumerate}
\item
to obtain clarity and uniformity in developing analysis with
infinitesimals by working in a version of set theory with a concept of
infinitesimal built in syntactically (see e.g., \cite{Ka04},
\cite{21d}), or
\item
to define nonstandard universes in a simplified way by providing
mathematical structures that would not have all the features of a
Hewitt--Luxemburg-style ultrapower \cite{He48}, \cite{Lu64}, but
nevertheless would suffice for basic application and possibly could be
helpful in teaching.
\end{enumerate}
The second approach has been carried out, for example, by James Henle
in his non-nonstandard Analysis \cite{He99} (based on
Schmieden--Laugwitz \cite{Sc58}; see also \cite{La01}), as well as by
Terry Tao in his ``cheap nonstandard analysis'' \cite{Ta}.  These
frameworks are based upon the identification of real sequences
whenever they are eventually equal.%
\footnote{I.e., modulo the Fr\'echet filter.}
A precursor of this approach is found in the work of Giuseppe Peano.
In his 1910 paper \cite{Pe10} (see also \cite{Pe57}) he introduced the
notion of the end (\emph{fine}; pl.\;\emph{fini}) of a function (i.e.,
sequence) based upon its eventual behavior.
%
%
Equality and order between such \emph{fini} are defined as follows:
\begin{itemize}
\item
$\text{fine}(f) = \text{fine}(g)$ if and only if~$f(n) = g(n)$
eventually;
\item 
$\text{fine}(f) < \text{fine}(g)$ if and only if~$f(n) < g(n)$
eventually.
\end{itemize}
The collection
\begin{equation}
\label{e11c}
P\hspace{-1pt}e=\big\{\text{fine}(f)
\colon\N\stackrel{f}{\rightarrow}\R\big\}
\end{equation}
extends the set of real numbers (included as constant sequences), is
partially ordered, and includes infinite and infinitesimal elements.
For instance, if~$f(n) = n$, then~$\text{fine}(f) > r$ for every~$r
\in \R$, since eventually~$f(n) > r$. Similarly, if~$g(n) = n^{-1}$,
then~$0<\text{fine}(g) < r$ for every positive~$r \in \R$, since
eventually~$0 < g(n) < r$.  Moreover, Peano defines operations on the
\emph{fini}.  For instance,~$\text{fine}(f) + \text{fine}(g)$ can be
defined as~$\text{fine}(f+g)$,
and~$\text{fine}(f)\cdot\text{fine}(g)$, as~$\text{fine}(fg)$.  The
construction results in Peano's partially ordered non-Archimedean
ring~$P\hspace{-1pt}e$ of~\eqref{e11c} with zero divisors that
extends~$\R$.  Commenting on Peano's 1910 construction, Fisher notes
that here 
\begin{quote}
Peano contradicts his contention of 1892, following Cantor, that
constant infinitesimals are impossible.  (Fisher \cite{Fi81}, 1981,
p.\;154)
\end{quote}
Peano's 1910 article seems to have been overlooked by Freguglia who
claims ``to put Peano's opinion about the unacceptability of the
actual infinitesimal notion into evidence'' \cite[p.\;145]{Fr21}.

\section{Refining the equivalence relation on Cauchy sequences}

The present text belongs to neither of the categories summarized in
Section~\ref{one}.  Rather, we propose to exploit a concept that is a
household word for most of the mathematical audience to a greater
extent than either Fr\'echet filters or ends of functions, namely
Cauchy sequences.  More precisely, we propose to factor the classical
homomorphism
\[
C\to\R,
\]
from the ring~$C$ of Cauchy sequences to its quotient space~$\R$,
through an intermediate integral domain~$D$, by refining the
traditional equivalence relation on~$C$.  The composition
\[
C\to D\to\R
\]
(``undoing'' the refinement) is the classical homomorphism.

The classical construction of the real line~$\R$ involves declaring
Cauchy sequences~$u=(u_i)$ and~$v=(v_i)$ to be equivalent if the
sequence~$(u_i-v_i)$ tends to zero.  Instead, we define an equivalence
relation~$\sim$ on~$C$ by setting~$u\sim v$ if and only if they
actually coincide on a \emph{dominant} set of indices~$i$ (this would
be true in particular if~$\text{fine}(u)=\text{fine}(v)$).  The notion
of dominance is relative to a nonprincipal ultrafilter~$\mathcal U$ on
the set of natural numbers.%
\footnote{An ultrafilter on $\N$ is a maximal collection of subsets
  of~$\N$ which is closed under finite intersection and passage to a
  superset.  An ultrafilter is nonprincipal if and only if it includes
  no finite subsets of~$\N$.}
The complement of a dominant set is called negligible.  Namely, we
have
\begin{equation}
\label{e11b}
u\sim v \text{ if and only if } \{ i\in\N \colon u_i=v_i\}\in \mathcal
U.
\end{equation}
We set~$D=C/\!\!\sim$.  Then a null sequence generates an
infinitesimal of~$D$.  By further identifying all null sequences with
the constant sequence~$(0)$, we obtain an epimorphism
\begin{equation}
\label{e11}
\mathbf{sh} \colon D\to\R,
\end{equation}
assigning to each Cauchy sequence, the value of its limit in~$\R$.  It
turns out that such a framework is sufficient to develop infinitesimal
analysis in the sense specified in Section~\ref{s2}.

\section{Null sequences, infinitesimals, and P-points}
\label{s2}

For the purposes of the definition below, it is convenient to
distinguish notationally between~$\N$ used as the index set of
sequences, and~$\omega$ used as the collection on which our filters
are defined.

\begin{definition}
A non-principal ultrafilter~$\mathcal U$ on~$\omega$ is called a
\emph{P-point} if for every partition~$\{ C_n \colon n\in\N \}$
of~$\omega$ such that
\[
(\forall n\in\N)\, C_n \not\in \mathcal U,
\]
there exists some~$A \in \mathcal U$ such that~$A \cap C_n$ is a
finite set for each~$n$.
\end{definition}

Informally,~$\mathcal U$ being a P-point means that every partition
of~$\omega$ into negligible sets is actually a partition into
\emph{finite} sets, up to a~$\mathcal U$-negligible subset
of~$\omega$.  The existence of P-point ultrafilters is consistent with
the Zermelo--Fraenkel set theory with the Axiom of Choice (ZFC) as is
their nonexistence, as shown by Shelah; see \cite{Wi}.  Their
existence is guaranteed by the Continuum Hypothesis by Rudin
\cite{Ru56}.

Theorem~\ref{t22} below originates with Choquet \cite{Ch68}; see Benci
and Di Nasso \cite[Proposition 6.6]{Be03}, and the end of this section
for additional historical remarks.

Denote by~$C\subseteq\R^\N$ the ring of Cauchy sequences (i.e.,
convergent sequences) of real numbers, and by~$D$ its quotient by the
relation~\eqref{e11b}.

\begin{theorem}
\label{t22}
$\mathcal U$ is a P-point ultrafilter if and only if~$D$ is isomorphic
to the ring of finite hyperreals in\,~$\R^\N\!/\mathcal U$ where the
epimorphism {\rm\textbf{sh}} of~\eqref{e11} corresponds to the
shadow/standard part.
\end{theorem}

\begin{proof}
Suppose every bounded sequence becomes Cauchy when restricted to a
suitable~$\mathcal U$-dominant set~$Y\subseteq\N$.  Let~$X_1= \N, X_2,
X_3, \ldots$ be an inclusion-decreasing sequence of sets in~$\mathcal
U$.  Define a sequence~$f$ by setting~$f(i) = \frac1n$ for~$i \in X_n
\setminus X_{n+1}$.  By hypothesis there is a~$Y\!\in\mathcal U$ such
that~$f$ is Cauchy on~$Y$.  Suppose the intersection~$Y \cap (X_n
\setminus X_{n+1})$ is infinite for some~$n$.  Then there are
infinitely many~$i$ where~$f(i) = \frac1n$ and also infinitely
many~$i$ where~$f(i) < \frac1{n+1}$ (this happens for all~$i \in Y
\cap X_{n+2} \in \mathcal U$).  Such a sequence is not Cauchy on~$Y$.
The contradiction shows that every intersection with~$Y$ must be
finite, proving that~$\mathcal U$ is a P-point.

Conversely, suppose~$\mathcal U$ is a P-point.  Let~$f$ be a bounded
sequence.  By binary divide-and-conquer, we construct inductively a
sequence of nested compact segments~$S_i\subseteq\R$ of length tending
to zero such that for each~$i$, the sequence~$f$ takes values in~$S_i$
for a dominant set~$X_i\subseteq\N$ of indices.  By hypothesis, there
exists~$Y\!\in\mathcal U$ such that each complement~$S_i \setminus
Y$ is finite.  Then the restriction of~$f$ to the dominant set~$Y$
necessarily tends to the intersection point~$\bigcap_{i\in \N}
S_i\in\R$.

Benci and Di Nasso give a similar argument relating \emph{monotone}
sequences and selective ultrafilters in \cite[p.\;376]{Be03}.
\end{proof}

If~$\mathcal U$ is not a P-point, the natural monomorphism~$D\to
\R^\N\!/\mathcal U$ will not be onto the finite part of the
ultrapower~$\R^\N\!/\mathcal U$.

Let~$I\subseteq D$ be the ideal of infinitesimals, and denote by
$I^{-1}$ the set of inverses of nonzero infinitesimals.

\begin{corollary}
If\,~$\mathcal U$ is a P-point ultrafilter, then the hyperreal field
$\R^\N\!/\mathcal U$ decomposes as a disjoint union~$D\sqcup I^{-1}$.
\end{corollary}

Here every element of~$I^{-1}$ can be represented by a sequence
tending to infinity.  In this sense, Cauchy sequences of reals
(together with a refined equivalence relation) suffice to construct
the hyperreal field~$\R^\ast=\R^\N\!/\mathcal U$ and enable analysis
with infinitesimals.  Thus, if~$f$ is a continuous real function on
$[0,1]$, its extension~$f^\ast$ to the hyperreal interval
$[0,1]^\ast\subseteq\R^\ast$ maps the equivalence class of a Cauchy
sequence~$(u_n)$ to the equivalence class of the Cauchy sequence
$(f(u_n))$.  The derivative of~$f$ at a real point~$x\in[0,1]$ is the
shadow
\[
\mathbf{sh}\left(\frac{f^\ast(x+\epsilon)-f^\ast(x)}{\epsilon}\right)
\]
for infinitesimal~$\epsilon\not=0$, while the integral
$\int_0^1f(x)dx$ is the shadow of the sum~$\sum_{i=1}^\mu
f^\ast(x_i)\epsilon$ where the~$x_i$ are the partition points of
$[0,1]^\ast$ into (an infinite integer)~$\mu$ subintervals, and
$\epsilon=\frac1\mu$.  For the definition of infinite Riemann sums see
Keisler \cite{Ke86}.

We elaborate on the history of the problem.  According to Cleave's
hypothesis \cite{Cl71}, every infinitesimal $\varepsilon\in\R^\ast$ is
determined by a sequence of real numbers tending to zero.  Cutland et
al.\;\cite{Cu88} showed that the Cleave hypothesis requires $\mathcal
U$ to be a P-point ultrafilter.  Benci and Di Nasso
\cite[p.\;376]{Be03} developed the problem further by relating the
Cleave hypothesis to the so-called selective ultrafilters and monotone
sequences.

\section{Concluding remarks}

\textbf{1.} Our approach can be seen as a formalisation of Cauchy's
sentiment that a null sequence ``becomes'' an infinitesimal, while a
sequence tending to infinity becomes an infinite number; see e.g.,
Bair et al.\;(\cite{20a}, 2020).

\textbf{2.}  A perspective on hyperreal numbers via Cauchy sequences
has value as an educational tool according to the post \cite{mse}.
Formally, one could start with the full ring of sequences (Cauchy or
not), and form the traditional quotient modulo a nonprincipal
ultrafilter, to obtain the standard construction of the hyperreal
field where every null sequence generates an infinitesimal (see e.g.,
\cite[p.\;913]{Ke86}, \cite{Bl14}; but there may be additional
infinitesimals).  However, pedagogical experience shows that the full
ring of sequences sometimes appears as a nebulous object to students
who are not yet familiar with the ultrafilter construction (see e.g.,
the discussion at \cite{mse}); starting with more familiar objects
such as Cauchy sequences builds upon the students' previous
experience, and may be more successful in facilitating intuitions
about infinitesimals.

\textbf{3.}  The refinement of the equivalence relation on Cauchy
sequences (modulo the existence of P-points) enables one to obtain the
(finite) hyperreals as the set of equivalence classes of Cauchy
sequences, and both equivalence classes and Cauchy sequences are
household concepts for any mathematics sophomore.  We still use
ultrafilters; however, in this setting, they are conceptually similar
to maximal ideals, routinely used in undergraduate algebra.

\textbf{4.}  The distinction between procedures and ontology is a key
related issue.  Historical mathematical pioneers from Fermat to Cauchy
applied certain procedures in their mathematics, and while modern
set-theoretic \emph{ontology} appears to be beyond their conceptual
world, many of the \emph{procedures} of modern mathematics are not
(see \cite{17d} for a more detailed discussion).  Arguably the
procedures in modern infinitesimal analysis are closer to theirs than
procedures in modern Weierstrassian analysis.  Note that the
mathematical pioneers would have been just as puzzled by Cantorian set
theory as by ultrafilters.  For example, Leibniz would have likely
rejected Cantorian set theory as incoherent because contrary to the
part-whole axiom; see \cite{21a}.

\textbf{5.}  Bertrand Russell accepted, as a matter of ontological
certainty, Cantor's position concerning the non-existence of
infinitesimals.  For an analysis of a dissenting opinion by
contemporary neo-Kantians see~\cite{13h}.

\textbf{6.}  Robinson's framework for analysis with infinitesimals is
the first (and currently the only) framework meeting the
Klein--Fraenkel criteria for a successful theory of infinitesimals in
terms of an infinitesimal treatment of the mean-value theorem and an
approach to the definite integral via partitions into infinitesimal
segments; see \cite{18i}.

\textbf{7.} Arguably, Robinson's framework is the most successful
theory of infinitesimals in applications to natural science,
probability, and related fields; see \cite{19d} as well as \cite{21b},
\cite{21c}.

\textbf{8.}  Robinson explained his choice of the name for his theory
as follows: ``The resulting subject was called by me Non-standard
Analysis since it involves and was, in part, inspired by the so-called
Non-standard models of Arithmetic whose existence was first pointed
out by T.\;Skolem'' \cite[p.\;vii]{Ro66}.

\textbf{9.}  Currently there are two popular approaches to Robinson's
mathematics: model-theoretic and axiomatic/syntactic.  Since Skolem's
construction \cite{Sk33} did not use either the Axiom of Choice or
ultrafilters, it is natural to ask whether one can develop an approach
to analysis with infinitesimals, meeting the Klein--Fraenkel criteria,
that does not refer to the notion of an ultrafilter at all.  The
answer is affirmative and was provided in \cite{21d} via an internal
axiomatic approach.  Hrbacek and Katz \cite{Hr21} present a
construction of Loeb measures and nonstandard hulls in internal set
theories.  The effectiveness of infinitesimal methods in analysis has
recently been explored in reverse mathematics; see e.g., Sanders
\cite{Sa20}.

\end{document}